\documentclass[review]{elsarticle}

\usepackage{lineno,hyperref}
\usepackage{graphicx}
\usepackage{amssymb}
\usepackage{epstopdf}
\usepackage{epsfig}
\usepackage{amsmath}
\usepackage{amsthm}
\newtheorem{definition}{Definition}
\newtheorem{theorem}{Theorem}
\newtheorem{lemma}{Lemma}
\newtheorem{remark}{Remark}
\newtheorem{example}{Example}
\newtheorem{corollary}{Corollary}
\makeatletter
\renewcommand{\maketag@@@}[1]{\hbox{\m@th\normalsize\normalfont#1}}%
\makeatother
\modulolinenumbers[5]
\journal{Applied Mathematical Modelling}

\begin{document}

\begin{frontmatter}

\title{Routh-Hurwitz criterion of stability and robust stability for fractional-order systems with order $\alpha\in\left [1,2  \right )$}

\author{Jing Yang}
\address{University of Electronic Science and Technology of China, Chengdu 611731,
China}

\author{Xiaorong Hou$^*$}
\address{University of Electronic Science and Technology of China, Chengdu 611731,
China}
\cortext[mycorrespondingauthor]{Corresponding author}
\ead{Houxr@uestc.edu.cn}

\author{Yajun Li}
\address{State Grid Xinyang Electric Power Supply Company, Xinyang,
China}

%

%

\begin{abstract}
Based on the generalized Routh-Hurwitz criterion, we propose a sufficient and necessary criterion for testing the stability of fractional-order linear systems with order $\alpha\in\left [1,2  \right )$, called the fractional-order Routh-Hurwitz criterion. Compared with the existing criterion, our one involves fewer and simpler expressions, which is significant for analyzing robust stability of fractional-order uncertain systems. All these expressions are explicit ones about the coefficients of the characteristic polynomial of system matrix, so the stable parameter region of fractional-order uncertain systems can be described directly. Some examples show the effectiveness of our method.
\end{abstract}


\end{frontmatter}

\linenumbers

\section{Introduction}
Routh-Hurwitz criterion is an effective method for analyzing the stability of integer-order linear time-invariant systems. In 2006, some sufficient or necessary stability conditions of fractional-order systems with order $\left (0,1  \right ]$ were given based on the integer-order Routh-Hurwitz criterion\cite{ahmed2006some}. Recently, for dimensions $n=2,3$, some sufficient and necessary stability conditions for the case of order $(0,2)$ were established; for dimension $n=4$, some sufficient or necessary stability conditions for the case of order $(0,2)$ were given\cite{bourafa2020some}. Some researchers proposed a sufficient and necessary stability condition of fractional-order systems with order $\alpha\in\left [1,2  \right )$\cite{tavazoei2009note} which transforms a n-dimension fractional-order system into a 2n-dimension integer-order system to analyze the stability. This method has beautiful form, but it has high computational complexity, especially for multi-parameter systems.
\par
For fractional-order uncertain systems, some sufficient and necessary methods have been proposed in recent years. For example, $\mu$-analysis method to test the robust stability of fractional-order uncertain systems with $\alpha\in\left [1,2  \right )$ can solve the robustness bound of perturbation parameters\cite{lu2013robust}. Moreover, based on cylindrical algebraic decomposition technique, some parameter space algorithms for analyzing the robust stability of fractional-order uncertain systems with $\alpha\in\left (0,2  \right )$ were proposed, with which the range of perturbation parameters can be obtained\cite{yang2019cad,yang2019robust}. When we consider the robust stability of multi-parameter systems, with these methods, all robust stability results can not be obtained directly but by very complicated calculations.
\par
In this paper, the problems of stability and robust stability on n-dimension fractional-order linear systems with $\alpha\in\left [1,2 \right )$ are considered. Based on the generalized Routh-Hurwitz criterion, a sufficient and necessary criterion for testing the stability of fractional-order linear systems with order $\alpha\in\left [1,2  \right )$, called the fractional-order Routh-Hurwitz criterion, is proposed. We also list complete, explicit expressions for $n=2,3,4$, respectively. Our criterion involves fewer and simpler expressions than the exiting method\cite{tavazoei2009note}. Meanwhile, all expressions in our method are explicit ones about the coefficients of the characteristic polynomial of system matrix. When we consider the robust stability of fractional-order systems with uncertain multiple parameters, our method can describe stability analysis results more easily than existing methods\cite{lu2013robust,yang2019cad,yang2019robust,2011Robust,2012Robust}.
\section{Preliminaries}
Let $f(z)$ be a complex coefficient polynomial and satisfy:
\begin{equation}
f(iz)=b_0z^n+b_1z^{n-1}+\cdots +b_n+i\left ( a_0z^n+a_1z^{n-1}+\cdots +a_n \right )\left ( a_0\neq 0 \right ),
\end{equation}
where $a_j(j=1,2,\cdots,n)$ and $b_j(j=1,2,\cdots,n)$ are real numbers.
\par
The $2n\times2n$ generalized Hurwitz matrix $H_f$ is constructed from $f(z)$ as follows:
\begin{equation}
H_f=\begin{bmatrix}
a_0 & a_1 & \cdots  & a_n & 0 & \cdots & 0\\
b_0 & b_1 & \cdots  & b_n & 0 & \cdots & 0\\
0 & a_0 & \cdots  & a_{n-1} & a_n & \cdots & 0\\
0 & b_0 & \cdots  & b_{n-1} & b_n & \cdots & 0\\
\vdots  & \vdots  & \vdots & \vdots  & \vdots  & \vdots &\vdots \\
0  & \cdots  & 0 & a_0  & \cdots  & a_{n-1} &a_n \\
0  & \cdots  & 0 & b_0  & \cdots  & b_{n-1} &b_n \\
\end{bmatrix}.
\end{equation}
\begin{lemma}\cite{gantmakher2000theory}(The Generalized Routh-Hurwitz Criterion)
\par All roots of $f(z)$ have negative real parts if and only if $\delta_{k}>0(k=1,2,\cdots,n)$, where $\delta_{k}(k=1,2,\cdots,n)$ is the $2k$-th order leading principle minor of $H_f$.
\end{lemma}
Consider the following n-dimension fractional-order linear time-invariant system:
\begin{equation}
\textrm{D}^{\alpha }x=Ax,
\end{equation}
where $\alpha \in [1,2)$ is fractional order, $A\in \mathbb{R}^{n\times n}$, $x=(x_1,x_2,\cdots,x_n)^T$ is state vector.
\begin{lemma}
\cite{matignon1996stability,matignon1998stability,petravs2011fractional}System(3) is asymptotically stable if and only if $\left |arg(\lambda_j )  \right |>\frac{\alpha\pi }{2}$, where $\lambda _{j}(j=1,2,\cdots ,n)$ are the eigenvalues of matrix $A$, $arg(\cdot )$ denotes the argument of a complex number.
\end{lemma}
Let
\begin{equation}
\begin{matrix}
\Omega :=\left \{ \gamma \in \mathbb{C}|\left |arg(\gamma)  \right |>\frac{\alpha \pi }{2} \right \}\\
\Sigma :=\left \{ \gamma\in \mathbb{C}|\left |arg(\gamma)  \right |<\frac{\alpha \pi }{2} \right \}\\
\Gamma :=\left \{ \gamma\in \mathbb{C}|\left |arg(\gamma)  \right |=\frac{\alpha \pi }{2} \right \}
\end{matrix},
\end{equation}
call them the stable region, the unstable region and the critical line of system(3), respectively(as shown in Figure 1).
\begin{figure}[h]
\begin{center}
\includegraphics[width=2.5in]{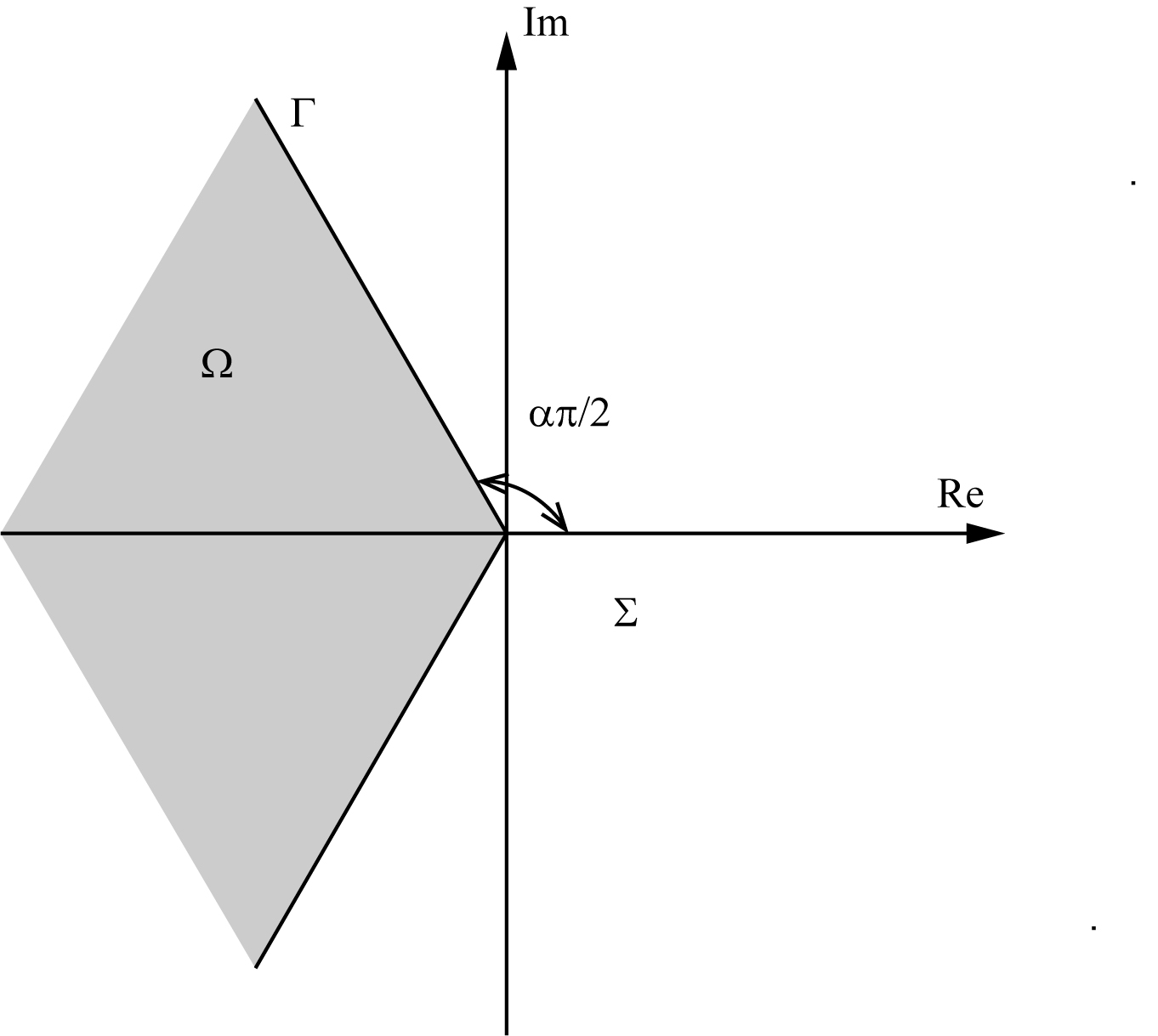}
\end{center}
\caption{$\Omega$, $\Sigma$ and $\Gamma$ of system(3)}
\label{fig1}
\end{figure}
\par
Suppose the characteristic polynomial of matrix $A$ is
\begin{equation}
f(\lambda)=\lambda^n+a_1\lambda^{n-1}+\cdots+a_n.
\end{equation}
So,
\begin{small}
\begin{equation}
f\left ( \lambda\cdot  e^{i\frac{\alpha\pi}{2} } \right )=\sum_{j=0}^{n}a_{j}\cdot cos\left ( \frac{(n-j)\cdot \alpha \pi }{2} \right )\cdot \lambda^{n-j}+i\cdot \left ( \sum_{j=0}^{n}a_{j}\cdot sin\left ( \frac{(n-j)\cdot \alpha \pi }{2}\right )\cdot \lambda^{n-j} \right ),
\end{equation}
\end{small}
where $a_0=1$.
\par
From Lemma 2, we know that system(3) is asymptotically stable if and only if all roots of $f(\lambda)$ are in stable region $\Omega$.
\par
In this paper, based on the generalized Routh-Hurwitz criterion, we propose a sufficient and necessary criterion, called the fractional-order Routh-Hurwitz criterion, to analyze the stability and robust stability of n-dimension fractional-order systems with order $1\leq \alpha <2$. Compared with the existing criterion, our one involves fewer and simpler expressions. All expressions in our results are explicit ones about the coefficients of the characteristic polynomial of system matrix, so the stable parameter region of fractional-order uncertain systems can be described directly.
\section{Main Results}
In this section, all notations are the same as above.
\par
\begin{definition}(The Fractional-Order Routh-Hurwitz Matrix)
\par For system(3), the $2n\times2n$ fractional-order Routh-Hurwitz matrix $H_{\alpha}$ of $f(\lambda)$ is defined as follows:
\begin{small}
\begin{equation}
H_{\alpha}=\left [ \begin{matrix}
a_0sin\left ( \frac{n\cdot \alpha \pi }{2} \right )  &a_{1}sin\left ( \frac{\left (n-1  \right )\cdot \alpha \pi }{2} \right )  & \cdots  & 0 & \cdots  &\cdots & 0\\
a_0cos\left ( \frac{n\cdot \alpha \pi }{2} \right )  &a_{1}cos\left ( \frac{\left (n-1  \right )\cdot \alpha \pi }{2} \right )  & \cdots  & a_n & \cdots  &\cdots & 0\\
0 & a_0sin\left ( \frac{n\cdot \alpha \pi }{2} \right ) &\cdots & a_{n-1}sin\left ( \frac{\alpha \pi }{2} \right )&\cdots &\cdots &0\\
0 & a_0cos\left ( \frac{n\cdot \alpha \pi }{2} \right ) &\cdots & a_{n-1}cos\left ( \frac{\alpha \pi }{2} \right )&a_n &\cdots &0\\
\vdots &\vdots &\vdots & \vdots &\vdots&\vdots &\vdots\\
0 &\cdots &\cdots & a_0sin\left ( \frac{n\cdot \alpha \pi }{2} \right ) &\cdots&a_{n-1}sin\left ( \frac{\alpha \pi }{2} \right ) &0\\
0 &\cdots &\cdots & a_0cos\left ( \frac{n\cdot \alpha \pi }{2} \right ) &\cdots&a_{n-1}cos\left ( \frac{\alpha \pi }{2} \right ) &a_n\\
\end{matrix} \right ].
\end{equation}
\end{small}
\end{definition}
\begin{theorem}(The Fractional-Order Routh-Hurwitz Criterion)
\par System(3) is asymptotically stable if and only if $\nabla_{p}> 0(p=1,2,\cdots,n)$, where $\nabla_{p}(p=1,2,\cdots,n)$ is the $2p$-th order leading principle minor of $H_{\alpha}$.
\end{theorem}
\begin{proof}
The coordinate system $xy$ counterclockwise turns through angle $\theta =\frac{\left (\alpha -1  \right )\pi }{2}$ as the coordinate system ${x}'{y}'$(as shown in Figure 2).
\begin{figure}[h]
\begin{center}
\includegraphics[width=2.5in]{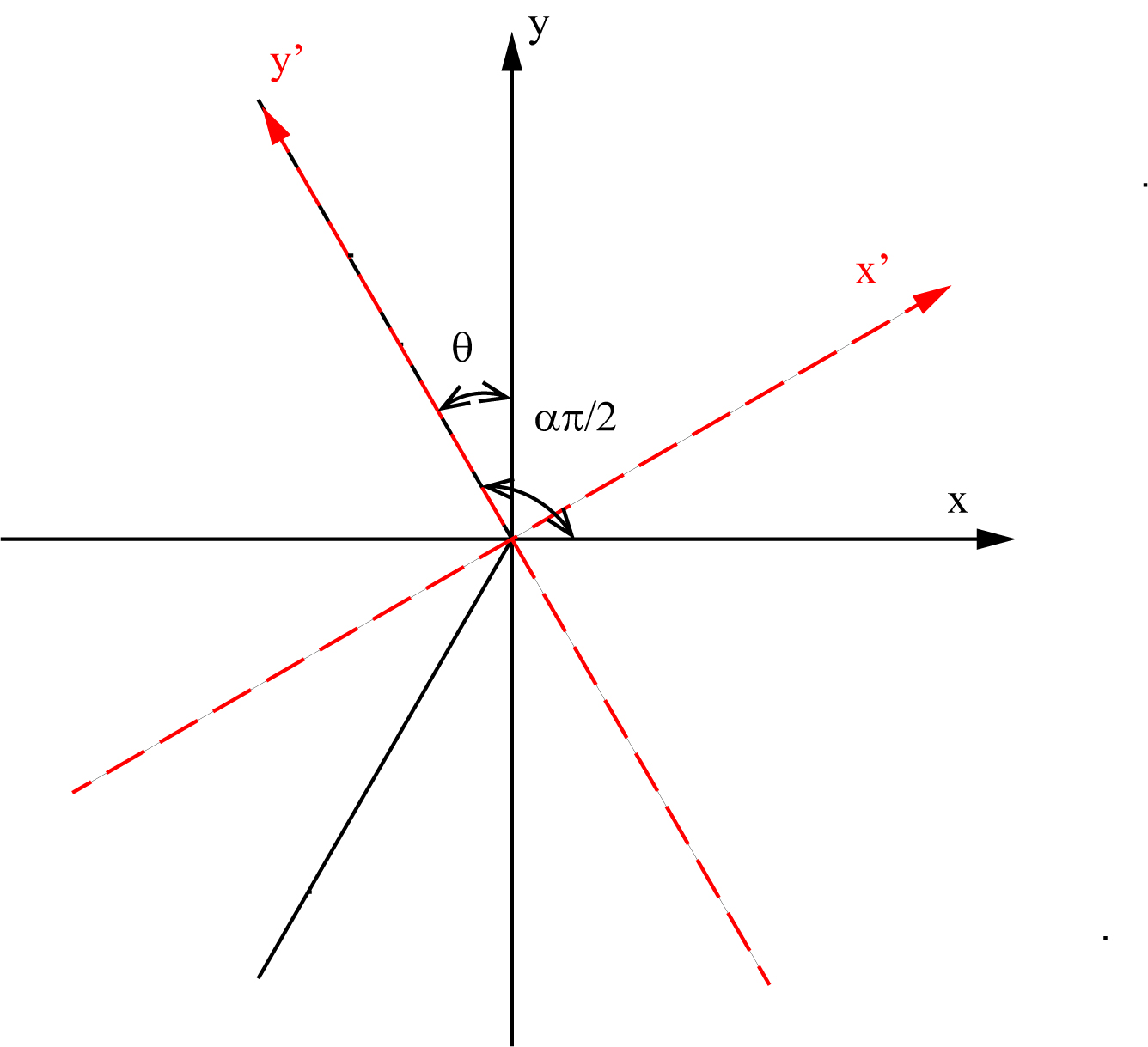}
\end{center}
\caption{Rotate the coordinate system}
\label{fig2}
\end{figure}
\par For system(3), in the new coordinate system ${x}'{y}'$, $f(\lambda)$ can be expressed as
\begin{equation}
g(\lambda)=f\left ( \lambda\cdot  e^{i\frac{\alpha -1}{2}\pi } \right ),
\end{equation}
thus
\begin{equation}
g(i\lambda)=f\left ( \lambda\cdot  e^{i\frac{\alpha\pi}{2} } \right ).
\end{equation}
Since $f(\lambda)$ is a real polynomial whose roots are symmetrical about the real axis in the coordinate system $xy$, its roots are in the stable region $\Omega$ if and only if they are in the left half plane of the coordinate system ${x}'{y}'$.
Based on the above analysis, according to Lemma 1 and Lemma 2, we have Theorem 1.
\end{proof}
\begin{remark}
Consider the general 2-dimension fractional-order system as follows:
\begin{equation}
\textrm{D}^{\alpha }x=Ax.
\end{equation}
where $A=\begin{bmatrix}
a_{11} & a_{12}\\
a_{21} & a_{22}
\end{bmatrix}$, $x=(x_1,x_2)^T$.
\par
System(10) is asymptotically stable if and only if the following integer-order system is asymptotically stable\cite{tavazoei2009note},
\begin{equation}
\dot{x}=\begin{bmatrix}
Asin\left ( \frac{\alpha \pi }{2} \right ) & Acos\left ( \frac{\alpha \pi }{2} \right )\\
-Acos\left ( \frac{\alpha \pi }{2} \right ) & Asin\left ( \frac{\alpha \pi }{2} \right )
\end{bmatrix}x.
\end{equation}
The characteristic polynomial of system(11) is
\begin{equation}
P(\lambda)=\lambda^4+a_1\lambda^3+a_2\lambda^2+a_3\lambda+a_4,
\end{equation}
where
\begin{equation}
\begin{aligned}
a_1 &=-2(a_{11}+a_{22})sin\left ( \frac{\alpha \pi }{2} \right ),\\
a_2 &= (a_{11}^{2}+4a_{11}a_{22}-2a_{12}a_{21}+a_{22}^2)sin^2\left ( \frac{\alpha \pi }{2} \right )+(a_{11}^{2}+2a_{12}a_{21}+a_{22}^{2})cos^2\left ( \frac{\alpha \pi }{2} \right ),\\
a_3 &= -2(a_{11}+a_{22})(a_{11}a_{22}-a_{12}a_{21})sin\left ( \frac{\alpha \pi }{2} \right ),\\
a_4 &= (a_{11}a_{12}-a_{12}a_{21})^2.
\end{aligned}
\end{equation}
The integer-order Routh-Hurwitz matrix $H_r$ of $P(\lambda)$ as follows:
\begin{equation}
H_r=\begin{bmatrix}
a_1 & a_3  & 0 & 0 \\
1 & a_2 & a_4  & 0 \\
0 & a_1 & a_3  & 0 \\
0 & 1 & a_2  & a_4 \\
\end{bmatrix}.
\end{equation}
Based on the integer-order Routh-Hurwitz criterion, we need to check $\Delta_{i}>0(i=1,2,3,4)$ to determine the stability of system(11), where $\Delta_{i}(i=1,2,3,4)$ is the $i$-th order leading principle minor of $H_r$ and
\begin{equation}
\begin{aligned}
\Delta_1 &=a_1,\\
\Delta_2 &=a_1a_2-a_3,\\
\Delta_3 &=-a_{1}^{2}a_4+a_1a_2a_3-a_{3}^{2}, \\
\Delta_4 &=-a_{4}(a_{1}^{2}a_4-a_1a_2a_3+a_{3}^{2}).
\end{aligned}
\end{equation}
The above integer-order Routh-Hurwitz criterion needs to calculate 4 leading principle minors that are complex.
\par
Consider the same fractional-order system(10), the characteristic polynomial of matrix $A$ is
\begin{equation}
f(\lambda)=\lambda^2-(a_{11}+a_{22})x+a_{11}a_{22}-a_{12}a_{21}.
\end{equation}
The fractional-order Routh-Hurwitz matrix $H_{\alpha}$ of polynomial(16) is as follows:
\begin{small}
\begin{equation}
H_{\alpha}=\begin{bmatrix}
sin\left ( \alpha \pi  \right ) & -\left ( a_{11}+a_{22} \right )sin\left ( \frac{\alpha \pi }{2} \right ) & 0 & 0\\
cos\left ( \alpha \pi  \right ) & -\left ( a_{11}+a_{22} \right )cos\left ( \frac{\alpha \pi }{2} \right ) & a_{11}a_{22}-a_{12}a_{21} & 0\\
0 &sin\left ( \alpha \pi  \right ) & -\left ( a_{11}+a_{22} \right )sin\left ( \frac{\alpha \pi }{2} \right ) & 0\\
0 & cos\left ( \alpha \pi  \right ) & -\left ( a_{11}+a_{22} \right )cos\left ( \frac{\alpha \pi }{2} \right ) & a_{11}a_{22}-a_{12}a_{21}
\end{bmatrix}.
\end{equation}
\end{small}
Based on Theorem 1, we only need to check two even-order leading principle minors $\nabla_{p}>0(p=1,2)$ of $H_{\alpha}$ to determine the stability of system(10), where
\begin{small}
\begin{equation}
\begin{aligned}
\nabla_1 &=-(a_{11}+a_{22})sin\left ( \frac{\alpha \pi }{2} \right ),\\
\nabla_2 &=(a_{22}a_{11}-a_{21}a_{12})sin^2\left ( \frac{\alpha \pi }{2} \right )\left ( (a_{11}+a_{22})^2-4(a_{22}a_{11}-a_{21}a_{12})cos^{2}\left ( \frac{\alpha \pi }{2} \right ) \right ).\\
\end{aligned}
\end{equation}
\end{small}
In this paper, $1\leq \alpha <2$, so $\nabla_{1}>0,\nabla_{2}>0$ is equivalent to
\begin{equation}
\begin{aligned}
\widetilde{\nabla}_1 &=-(a_{11}+a_{22})>0,\\
\widetilde{\nabla}_2 &=(a_{22}a_{11}-a_{21}a_{12})\left ( (a_{11}+a_{22})^2-4(a_{22}a_{11}-a_{21}a_{12})cos^{2}\left ( \frac{\alpha \pi }{2} \right ) \right )>0.\\
\end{aligned}
\end{equation}
Compared with the existing method, our one involves smaller numbers of leading principle minors. Each $\nabla_{p}$ is simpler than $\Delta_{i}$, $i=2p$. Our method has less computational complexity than the existing method\cite{tavazoei2009note}. Especially for fractional-order uncertain systems, the advantage of low computational complexity is significant.
\end{remark}
For system(3), suppose the characteristic polynomial of $A$ is $f(\lambda)=\lambda^n+a_1\lambda^{n-1}+\cdots+a_n$. Since systems with dimensions $n=2,3,4$ are often used, based on Theorem 1, we have the following corollaries. In the following, always set $cos^2\left ( \frac{\alpha \pi }{2} \right )=s$.
\begin{corollary}
In the case of $n=2$, system(3) is asymptotically stable if and only if
\begin{equation}
 a_{1}>0, \quad a_2\left ( a_1^2-4a_2s \right )>0
\end{equation}
\end{corollary}
\begin{corollary}
In the case of $n=3$, system(3) is asymptotically stable if and only if
\begin{equation}
\begin{aligned}
 &a_1>0, \quad \left ( 4a_1a_3-4a_2^2 \right )s+a_1^2a_2-a_1a_3 >0,\\
 &a_3\cdot(64a_3^2s^3-\left ( 16a_1a_2a_3+48a_3^2 \right )s^2+( 4a_1^3a_3-4a_1a_2a_3+4a_2^3+12a_3^2 )s\\&-a_1^2a_2^2+2a_1a_2a_3-a_3^2)>0.
\end{aligned}
\end{equation}
\end{corollary}
\begin{corollary}
In the case of $n=4$, system(3) is asymptotically stable if and only if
\begin{equation}
\begin{aligned}
 &a_1>0, \quad \left ( 4a_1a_3-4a_2^2 \right )s+a_1^2a_2-a_1a_3>0,\\
  &\left ( 64a_1a_4^2-128a_2a_3a_4+64a_3^3 \right )s^3-(16a_1^2a_3a_4+16a_1a_2^2a_4+16a_1a_2a_3^2-64a_1a_4^2\\
 &+96a_2a_3a_4-48a_3^3 )s^2+(4a_1^3a_2a_4-4a_1^3a_3^2+8a_1^2a_3a_4+4a_1a_2^2a_4+4a_1a_2a_3^2-4a_2^3a_3\\
 &-16a_1a_4^2+16a_2a_3a_4-12a_3^3)s-a_1^3a_2a_4+a_1^2a_2^2a_3+a_1^2a_3a_42a_1a_2a_3^2+a_3^3>0,\\
 &a_4\cdot(4096a_4^3s^6+(-1024a_1a_3a_4^2-8192a_4^3)s^5+(256a_1^2a_2a_4^2+1536a_1a_3a_4^2-512a_2^2a_4^2
\\
    &+256a_2a_3^2a_4+6144a_4^3)s^4+(-64a_1^4a_4^2-64a_1^2a_2a_4^2-64a_1a_2^2a_3a_4-1024a_1a_3a_4^2
\\
    &+512a_2^2a_4^2-64a_2a_3^2a_4-64a_3^4-2048a_4^3)s^3+(48a_1^4a_4^2+16a_1^3a_2a_3a_4-64a_1^2a_2a_4^2
\\
    &-16a_1^2a_3^2a_4-32a_1a_2^2a_3a_4+16a_1a_2a_3^3+16a_2^4a_4+384a_1a_3a_4^2-128a_2^2a_4^2-64a_2a_3^2a_4\\
    &+48a_3^4+256a_4^3)s^2+(-12a_1^4a_4^2+4a_1^3a_2a_3a_4-4a_1^3a_3^3-4a_1^2a_2^3a_4+16a_1^2a_2a_4^2+8a_1^2a_3^2a_4\\
    &+16a_1a_2^2a_3a_4+4a_1a_2a_3^3-4a_2^3a_3^2-64a_1a_3a_4^2+16a_2a_3^2a_4-12a_3^4)s+a_1^4a_4^2-2a_1^3a_2a_3a_4\\
    &+a_1^2a_2^2a_3^2+2a_1^2a_3^2a_4-2a_1a_2a_3^3+a_3^4)>0.
\end{aligned}
\end{equation}
\end{corollary}
The fractional-order Routh-Hurwitz criterion can be used to analyze the robust stability of fractional-order uncertain systems.
\par
Consider fractional-order uncertain system as follows:
\begin{equation}
\textrm{D}^{\alpha }x=A(\beta)x,
\end{equation}
where $(\alpha,\beta)=(\alpha,\beta_1,\beta_2,\cdots,\beta_i)$ are uncertain parameters, $x=(x_1,x_2,\cdots,x_n)^T$ is the state vector.
\par
The characteristic polynomial of matrix $A(\beta)$ is
\begin{equation}
f(\lambda;\alpha,\beta)=\lambda^n+a_1(\alpha,\beta)\lambda^{n-1}+\cdots+a_n(\alpha,\beta).
\end{equation}
Since all expressions in our method are explicit ones about the coefficients of the characteristic polynomial of system matrix, so Theorem 1 is also effective for analyzing the robust stability of system(23).
\par
System(23) is of certain parameters for given parameter $(\alpha,\beta)$. We call the parameter $(\alpha,\beta)$ a stable parameter if the corresponding system is asymptotically stable. The set of all stable parameters is called the stable parameter region, denoted by $\Sigma(\alpha,\beta)$. According to Theorem 1, the stable parameter region of system(23) is the set of the solutions of $\left \{ \nabla_p>0,p=1,2,\cdots,n \right \}$. All expressions in our results are explicit ones about the coefficients of the characteristic polynomial of system matrix, so the stable parameter region can be described directly.
\section{Illustrative Examples}
\begin{example}
\cite{yang2019cad}Consider the following fractional-order uncertain system:
\begin{equation}
\frac{\mathrm{d}^{1.5 }x(t) }{\mathrm{d} t^{1.5 }}=(A_0+\beta_1A_1+\beta_2A_2)x(t),
\end{equation}
where
\begin{center}
$A_0=\begin{bmatrix}
-1 & 3\\
0 & -1
\end{bmatrix},
A_1=\begin{bmatrix}
-1 & 0\\
1 & -1
\end{bmatrix},
A_2=\begin{bmatrix}
-1 & 1\\
0 & 0
\end{bmatrix}.$
\end{center}
\begin{figure}[h]
\begin{center}
\includegraphics[width=3in]{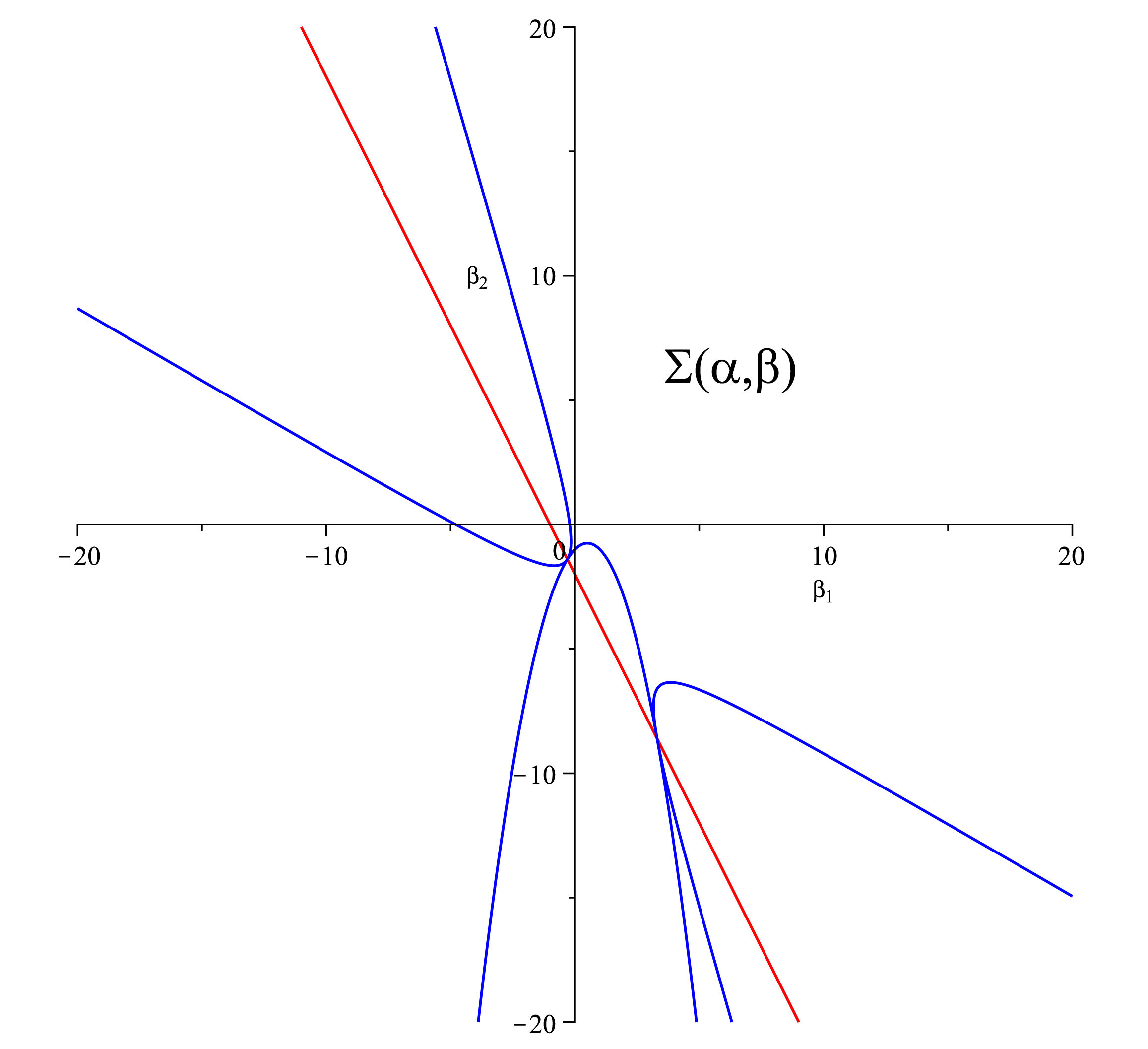}
\end{center}
\caption{The solutions of the set of inequalities(27)}
\label{fig3}
\end{figure}
The characteristic polynomial of $A_0+\beta_1A_1+\beta_2A_2$ is
\begin{equation}
f(\lambda)=\lambda^2+(2\beta_1+2+\beta_2)\lambda+\beta_1^2+\beta_2-\beta_1+1,
\end{equation}
Based on Corollary 1, we know that system(25) is asymptotically stable if and only if
\begin{equation}
\begin{aligned}
 &2\beta_1+2+\beta_2>0,\\
 &(\beta_1^2+\beta_2-\beta_1+1)\left ( (2\beta_1+2+\beta_2)^2-2(\beta_1^2+\beta_2-\beta_1+1)\right )>0.
\end{aligned}
\end{equation}
Parameters that satisfy the set of inequalities(27) are stable parameters of system(25). The solutions of the set of inequalities(27) are shown in Figure 3, in which the stable parameter region is marked.
\begin{remark}
System(25) has been considered\cite{lu2013robust,yang2019cad}. By using existing methods, the robustness bound of a single parameter can be obtained, or the stable parameter region can be determined by taking points to test the stability of some corresponding systems. All expressions in our results are explicit ones about the coefficients of the characteristic polynomial of system matrix, the relationship among multiple parameters can be described and the stable parameter region can be solved directly.
\end{remark}
\end{example}
\begin{example}
Consider a fractional-order system with $n=3$, where $\alpha\in\left [1,2  \right )$ and $\beta\in(0,10)$ are uncertain parameters. Suppose the characteristic equation of the system matrix is:
\begin{equation}
\lambda^3+(\beta-\alpha)\lambda^2+2\beta\lambda+4=0.
\end{equation}
\begin{figure}[h]
\begin{center}
\includegraphics[width=2.6in]{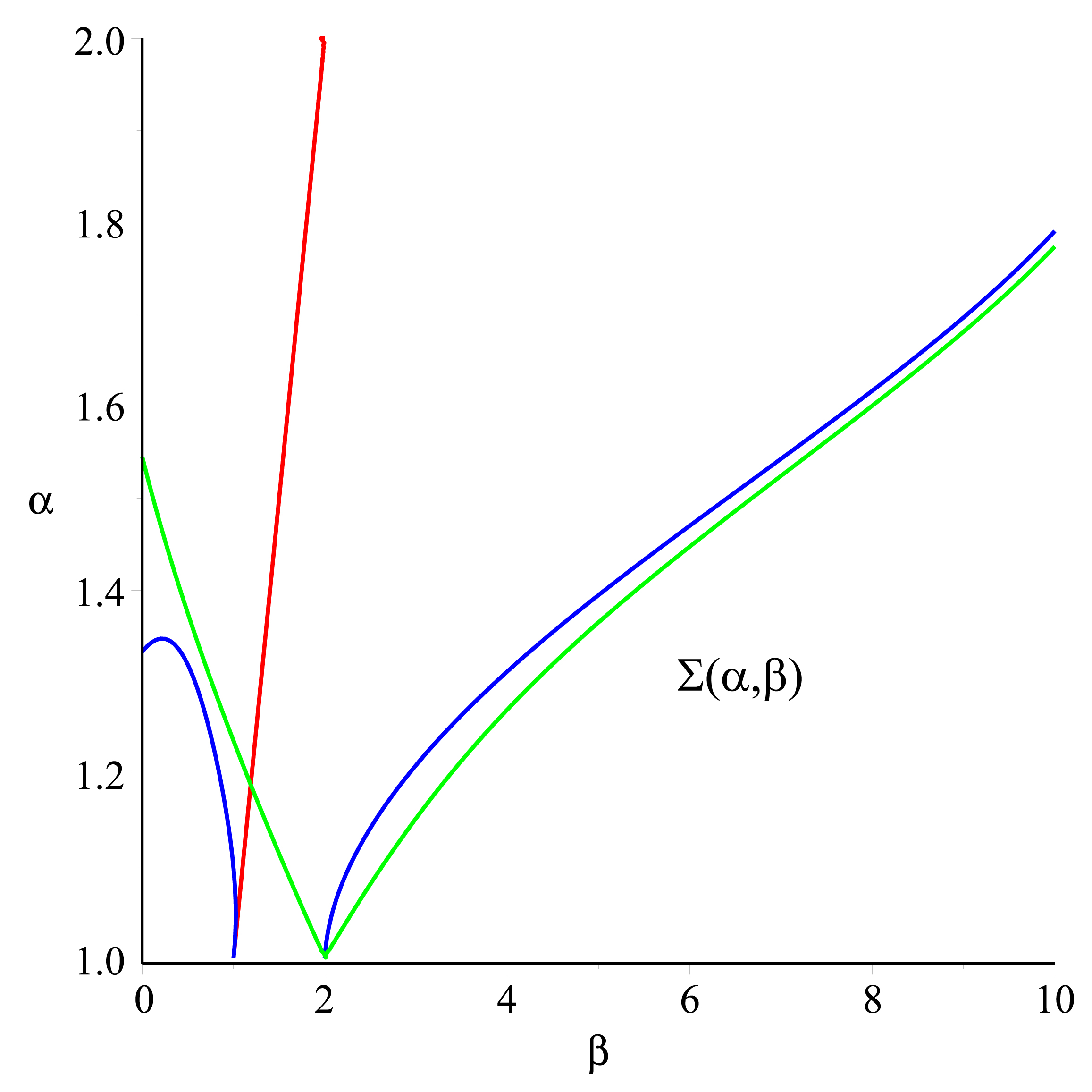}
\end{center}
\caption{The solutions of the set of inequalities(29)}
\label{fig4}
\end{figure}
\par Based on Corollary 2, we know that the system is asymptotically stable if and only if
\begin{equation}
\begin{aligned}
 &\beta-\alpha>0,\\
 &\left (16(\beta-\alpha)-16\beta^2  \right )s+2(\beta-\alpha)^2\beta-4\beta+4\alpha>0,\\
 &-4(1024s^3-\left (128(\beta-\alpha)\beta+768  \right )s^2+(16(\beta-\alpha)^3-32(\beta-\alpha)\beta\\
 &+32\beta^3+192)s-4(\beta-\alpha)^2\beta^2+16(\beta-\alpha)\beta-16)>0.
\end{aligned}
\end{equation}
The solutions of the set of inequalities(29) are shown in Figure 4, in which the stable parameter region is marked.
\end{example}
\begin{remark}
Consider fractional-order systems with uncertain order, the existing methods can not describe the relationship between order parameter and stability directly\cite{yang2019robust,2011Robust,2012Robust}. By using our method, systems with uncertain order and uncertain other parameters can be analyzed easily and all results are explicit expressions about the coefficients of the characteristic polynomial of system matrix.
\end{remark}
\begin{example}
\cite{lu2013stability}Consider the following fractional-order uncertain system with $n=4$:
\begin{equation}
\frac{\mathrm{d}^{1.5 }x(t) }{\mathrm{d} t^{1.5 }}=(\beta_1A_1+\beta_2A_2)x(t)
\end{equation}
where $\beta_1+\beta_2=1,\beta_{1},\beta_{2}\geq 0$, $A_1=\Gamma-\varepsilon b*c$, $A_2=\Gamma+\varepsilon b*c$,
\begin{center}
$\Gamma =\begin{bmatrix}
-1 & 1 & 0 & 0\\
0 & -2 & 1 & 0\\
0 & 0 & -3 & 1\\
-10 & -10 & -20 & -6
\end{bmatrix},
b =\begin{bmatrix}
0\\
0 \\
0\\
1
\end{bmatrix},
c'=\begin{bmatrix}
3\\
3\\
0\\
0
\end{bmatrix}.$
\end{center}
\begin{figure}[h]
\begin{center}
\includegraphics[width=2.6in]{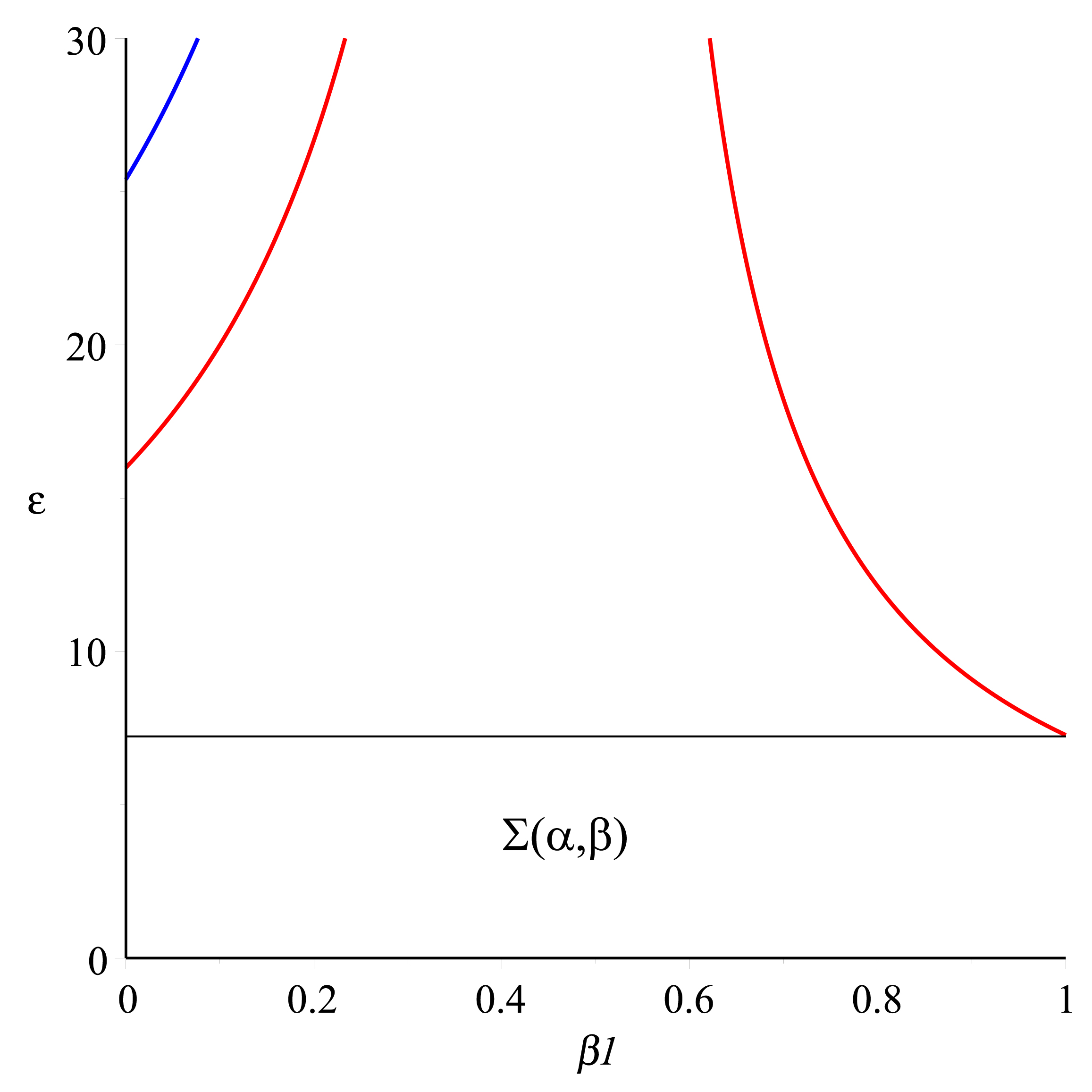}
\end{center}
\caption{The solutions of the set of inequalities}
\label{fig3}
\end{figure}
Let $\beta_2=1-\beta_1$, the characteristic polynomial of system matrix is:
\begin{equation}
f(\lambda)=\lambda^4+12\lambda^3+67\lambda^2+\left (6\beta_1\varepsilon-3\varepsilon+142\right )\lambda+12\beta_1\varepsilon-6\varepsilon+96,
\end{equation}
based on Corollary 3, the solutions of the set of inequalities are shown in Figure 5, the stable parameter region of system(30) is marked. For system(30), since $\beta_1$ satisfies $0\leq\beta_1<1$, which means the set of inequalities has solutions for all $\beta_1\in \left [0,1 \right )$, so the maximum value of $\varepsilon$ is 7.274.
\begin{remark}
The above examples showed the effectiveness of our method for analyzing fractional-order systems with $n=2,3,4$. For general n, based on the existing method\cite{tavazoei2009note}, we need to test $\Delta_{i}>0,i=1,2,\cdots,2n$ to determine the stability of system(23). And as a comparison in this case, using our method, we only need to test whether $\nabla_{p}>0,p=1,2,\cdots,n$ to analyze the stability. The number of leading principle minors of our method is smaller than existing methods.
\end{remark}
\end{example}

\section{Conclusions}
Based on the generalized Routh-Hurwitz criterion, we give a fractional-order Routh-Hurwitz criterion for analyzing the stability and robust stability of fractional-order linear systems with $1\leq\alpha<2$. Compared with existing methods, our one involves fewer and simpler expressions, so it has low computational complexity. All expressions in our results are explicit ones about the coefficients of the characteristic polynomial of system matrix, so the stable parameter region of fractional-order uncertain systems can be described directly. Our method is suitable for some complex cases just like systems with uncertain order and uncertain other parameters.
\section*{Acknowledgment}
This work was partially supported by the National Natural Science Foundation of China
under Grants No.12171073.
\section*{References}
\bibliographystyle{elsarticle-num-names}
\biboptions{square,numbers,sort&compress}
\bibliography{elsarticle-num-names}

\begin{thebibliography}{10}
\expandafter\ifx\csname url\endcsname\relax
  \def\url#1{\texttt{#1}}\fi
\expandafter\ifx\csname urlprefix\endcsname\relax\def\urlprefix{URL }\fi
\expandafter\ifx\csname href\endcsname\relax
  \def\href#1#2{#2} \def\path#1{#1}\fi

\bibitem{ahmed2006some}
A.~E.-S. E.~Ahmed, H.~A. El-Saka, On some routh-hurwitz conditions for
  fractional order differential equations and their applications in lorenz,
  r{\"o}ssler, chua and chen systems, Physics Letters A 358~(1) (2006) 1--4.

\bibitem{bourafa2020some}
A.~M. S.~Bourafa, M-S.~Abdelouahab, On some extended routh-hurwitz conditions
  for fractional-order autonomous systems of order $\alpha\in(0,2)$ and their
  applications to some population dynamic models, Chaos, Solitons and Fractals
  133 (2020) 109623.

\bibitem{tavazoei2009note}
M.~Tavazoei, M.~Haeri, A note on the stability of fractional order systems,
  Mathematics and Computers in simulation 79~(5) (2009) 1566--1576.

\bibitem{lu2013robust}
Y.~C. J.G.~Lu, W.~Cheng, Robust asymptotical stability of fractional-order
  linear systems with structured perturbations, Computers and Mathematics with
  Applications 66~(5) (2013) 873--882.

\bibitem{yang2019cad}
X.~H. J.~Yang, M.~Luo, A cad-based algorithm for solving stable parameter
  region of fractional-order systems with structured perturbations, Fractional
  Calculus and Applied Analysis 22~(2) (2019) 509--521.

\bibitem{yang2019robust}
J.~Yang, X.~Hou, Robust bounds for fractional-order systems with uncertain
  order and structured perturbations via cylindrical algebraic decomposition
  method, Journal of the Franklin Institute 356~(7) (2019) 4097--4105.

\bibitem{2011Robust}
L.~L.Zeng, P.Cheng, W.Yong, Robust stability analysis for a class of fractional
  order systems with uncertain parameters, Journal of the Franklin Institute
  348~(6) (2011) 1101--1113.

\bibitem{2012Robust}
C.Li, J.Wang, Robust stability and stabilization of fractional order interval
  systems with coupling relationships: The $0<\alpha<1$ case, Journal of the
  Franklin Institute 349~(7) (2012) 2406--2419.

\bibitem{gantmakher2000theory}
F.~Gantmakher, The Theory of Matrices, Volume 2, Vol. 133, American
  Mathematical Soc., 2000.

\bibitem{matignon1996stability}
D.Matignon, Stability results for fractional differential equations with
  applications to control processing, in: Computational engineering in systems
  applications, Vol.~2, Citeseer, 1996, pp. 963--968.

\bibitem{matignon1998stability}
D.Matignon, Stability properties for generalized fractional differential
  systems, in: ESAIM: proceedings, Vol.~5, EDP Sciences, 1998, pp. 145--158.

\bibitem{petravs2011fractional}
I.~Petr{\'a}{\v{s}}, Fractional-order nonlinear systems: modeling, analysis and
  simulation, Springer Science and Business Media, 2011.

\bibitem{lu2013stability}
J.~Lu, Y.~Chen, Stability and stabilization of fractional-order linear systems
  with convex polytopic uncertainties, Fractional Calculus and Applied Analysis
  16~(1) (2013) 142--157.

\end{thebibliography}

\end{document}